\documentclass{amsart}
\usepackage{latexsym,amsmath,amssymb}
\usepackage{amsthm}
\usepackage{pslatex}
\usepackage{cite}
\topmargin
5pt\footskip 35pt

\newtheorem{theorem}{Theorem}[section]
\theoremstyle{plain}

\newtheorem{corollary}[theorem]{Corollary}

\newtheorem{lemma}[theorem]{Lemma}

\newtheorem{proposition}[theorem]{Proposition}
\newtheorem{remark}[theorem]{Remark}
\numberwithin{equation}{section}

\def\nn{\nonumber}

\def\fa1{f_\alpha}

\def\dbp#1#2{\mathcal{B}_{#1;q}(#2)}
\def\dbn#1{\mathcal{B}_{#1;q}}
\def\dbu{\mathcal{B}_q}
\def\deu{\mathcal{E}_q}
\def\dep#1#2{\mathcal{E}_{#1;q}(#2)}
\def\den#1{\mathcal{E}_{#1;q}}
\def\dgu{\mathcal{G}_q}
\def\dgp#1#2{\mathcal{G}_{#1;q}(#2)}
\def\dgn#1{\mathcal{G}_{#1;q}}

\def\eqT{e_q}

\def\cnq#1{c_{#1;\,q}}
\def\binomcq#1#2{\binom{#1}{#2}_{\!\!\!c;\,q}}
\def\qplus{\oplus_q}

\def\Dcq#1{D_{c_q;\,#1}}
\def\Dcqk#1#2{D_{c_q;\,#1}^{\,#2}}
\def\nDq#1{D_{q;#1}}
\def\dcq{d_{c_q}}
\def\Icq#1{I_{c_q;\,#1}}
\def\intcq#1#2{\int #1 \dcq#2}
\def\lla{\left\langle}
\def\rra{\right\rangle}

\begin{document}
\title[New degenerated 
polynomials]{New degenerated 
polynomials arising from non-classical Umbral Calculus}
\author{Orli Herscovici}
\address{Department of Mathematics, University of Haifa, 3498838  Haifa, Israel}
\email{orli.herscovici@gmail.com}

\author{Toufik Mansour}
\address{Department of Mathematics, University of Haifa, 3498838  Haifa, Israel}
\email{tmansour@univ.haifa.ac.il}

\begin{abstract}
We introduce new generalizations of the Bernoulli, Euler, and
Genocchi polynomials and numbers based on the Carlitz-Tsallis 
degenerate exponential function and concepts of the Umbral 
Calculus associated with it. Also, we present generalizations of 
some familiar identities and connection between these kinds of 
Bernoulli, Euler, and Genocchi polynomials. Moreover, we establish 
a new analogue of the Euler identity for the degenerate Bernoulli 
numbers.
\end{abstract}
\maketitle

\section{Introduction}
The Bernoulli, Euler, and Genocchi numbers and polynomials are 
closely connected to each other \cite{Dilcher2016}. Their study 
attracts attention of many researchers (see \cite{Araci2014, He2015,
Ozden2010, Simsek2010} and reference therein). Besides the classical 
versions there exist their different $q$-analogues  and parameterized 
versions \cite{Acikgoz2015, Acikgoz2009, He2015, Simsek2010}. The 
degenerate versions of the Bernoulli and Euler numbers were defined 
and studied by Carlitz. They are based on a degenerate exponential 
function $e_{\lambda,\mu}(t)=(1+\lambda t)^\mu$. In this way the 
degenerate Bernoulli numbers of Carlitz are defined by a generating 
function
\begin{align}
\frac{t}{(1+\lambda t)^\mu-1}=\sum_{n=0}^\infty\beta_n(\lambda)
\frac{t^n}{n!},
\end{align}
with condition $\lambda\mu=1$.
Degenerate versions of the Bernoulli, Euler, and Genocchi polynomials 
were studied by different researchers (see \cite{Lim2016, Young2008} 
and references therein). 
The umbral calculus was one of the methods used for study of these 
polynomials. However, those degenerate polynomials were defined in 
terms of versions exponential functions with classical additive property
and studied by techniques of the classical umbral calculus.
For example, the degenerate Bernoulli polynomials studied in 
\cite{Young2008} are defined as
\begin{align}
\frac{t}{(1+\lambda t)^\frac{1}{\lambda}-1)}(1+\lambda t)^\frac{x}{\lambda}
=\sum_{n=0}^\infty\beta_n(\lambda,x)\frac{t^n}{n!},
\end{align}
where the degenerate Bernoulli numbers are evaluated, as usually, as 
$\beta_n(\lambda)=\beta_n(\lambda,0)$.
In this work we define a new degenerate Bernoulli, Euler, and Genocchi 
polynomials and numbers and study them by applying non-classical 
umbral calculus. They are based on degenerated, or, in other words, 
deformed exponential function that their additive property is deformed 
too. Our results generalize many of well known identities for classical 
case. Moreover, we bring a new analogue of the Euler identity for the 
Bernoulli numbers and establish connections between degenerate 
versions of the Bernoulli, Euler, and Genocchi polynomials.

This paper is organized as following.
We start from some definitions and useful theorems of umbral calculus.
Each one from the following three sections considers, respectively, the 
degenerate Bernoulli, Euler, and Genocchi polynomials and numbers. 
The last section considers connections between these polynomials and 
shows a way to find connections with other polynomials.

\section{Background and definitions}
Let us consider the umbral calculus associated with the
deformed exponential function
\begin{align}
\eqT(x)=(1+(1-q)x)^\frac{1}{1-q}, \label{DefNextqExp}
\end{align}
defined by Carlitz in \cite{Carlitz1956} with substitution $\lambda=1-q$
and by Tsallis in \cite{Tsallis1988}. It is easy to see that this Carlitz-Tsallis 
exponential
function \eqref{DefNextqExp} is an eigenfunction of the
operator
$\nDq{x}\equiv(1+(1-q)x)\frac{d}{dx}$,
where $d/dx$ is the ordinary Newton's derivative.
This exponential function has the
following extension as formal power series (see \cite{Borges1998}):
\begin{align*}
\eqT(x)=1+\sum_{n=1}^\infty Q_{n-1}(q)\frac{x^n}{n!},\nn
\end{align*}
where $Q_{n}(q)=1\cdot q(2q-1)\ldots(nq-(n-1))$. Let us
define a sequence $\cnq{n}$ as following
\begin{align*}
\cnq{n}=\left\{\begin{array}{ll}\frac{n!}{Q_{n-1}(q)}, & n\geq 1,\\
1, & n=0.\end{array} \right.
\end{align*}
Therefore, by using this notation, we can
write
\begin{align}
\eqT(x)=\sum_{n=0}^\infty\frac{x^n}{\cnq{n}}.\label{Def::DegExpSer}
\end{align}
Define $x\qplus y=x+y+(1-q)xy$ (see, Borges \cite{Borges2004}) and
$(x+y)^n_c=\sum_{k=0}^n\frac{\cnq{n}}{\cnq{k}\cnq{n-k}}x^ky^{n-k}$.
Now we can state the following Proposition.
\begin{proposition}
For any $x,y\in\mathbb{C}$ it holds that
\begin{align*}
\sum_{n=0}^\infty\frac{(xt\qplus yt)^n}{\cnq{n}}=
\sum_{n=0}^\infty(x+y)^n_c\frac{t^n}{\cnq{n}}.
\end{align*}
\end{proposition}
\begin{proof}
It follows from \eqref{Def::DegExpSer} that
\begin{align}
\eqT(xt)\eqT(yt)&=\sum_{n=0}^\infty\frac{x^nt^n}{\cnq{n}}\cdot
\sum_{k=0}^\infty\frac{y^kt^k}{\cnq{k}}
=\sum_{n=0}^\infty\sum_{k=0}^n\frac{\cnq{n}}{\cnq{k}\cnq{n-k}}x^ky^{n-k}
\frac{t^n}{\cnq{n}}\nn\\
&=\sum_{n=0}^\infty(x+y)^n_c\frac{t^n}{\cnq{n}}=\eqT((x+y)_ct),
\label{p5eq10}
\end{align}
From another side we have
\begin{align}
\eqT(xt)\eqT(yt)&=(1+(1-q)xt)^\frac{1}{1-q}(1+(1-q)yt)^\frac{1}{1-q}
=\left(1+(1-q)(xt+yt+(1-q)xyt^2)\right)^\frac{1}{1-q}\nn\\
&=\eqT(xt\qplus yt)=\sum_{n=0}^\infty\frac{(xt\qplus yt)^n}{\cnq{n}},
\label{p5eq11}
\end{align}
Therefore, by comparing \eqref{p5eq10} with \eqref{p5eq11}, we obtain 
that
$\eqT(xt\qplus yt)=\eqT((x+y)_ct)$, or, in more detailed notation,
\begin{align*}
\sum_{n=0}^\infty\frac{(xt\qplus yt)^n}{\cnq{n}}=
\sum_{n=0}^\infty(x+y)^n_c\frac{t^n}{\cnq{n}},
\end{align*}
which completes the proof.
\end{proof}

From here and on we will define
$\binomcq{n}{k}=\frac{\cnq{n}}{\cnq{k}\cnq{n-k}}$.
Clearly, $(x+y)^n_c=\sum_{k=0}^n\binomcq{n}{k}x^ky^{n-k}$.\medskip

Let $P$ be the algebra of polynomials
in a single variable $x$ over the field $\mathbb{C}$ and $P^*$ be the 
vector space of all linear functionals on $P$. The notation
$\left\langle L|p(x)\right\rangle $ 
denotes the action of a linear
functional $L$ on a polynomial $p(x)$, and the vector space
operations on $P^*$ are defined by
$\left\langle \alpha_1L_1+\alpha_2L_2|p(x)\right\rangle 
=\alpha_1\left\langle L_1|p(x)\right\rangle +\alpha_2\left\langle L_2|p(x)
\right\rangle$ for any constants $ \alpha_1,\alpha_2\in\mathbb{C}.$
Let $\mathcal{F}$ denote the algebra of formal power series in a
single variable $t$ over the field $\mathbb{C}$:
$$\mathcal{F}=\left\lbrace f(t)=\sum\limits_{k\geq 0}a_k\frac{t^k}{\cnq{k}}
\quad\Big|\quad a_k\in\mathbb{C}\right\rbrace.$$
The formal power series $f(t)$ defines a linear functional on $P$
by setting 
\begin{equation}
\left\langle f(t)|x^n\right\rangle =\cnq{n}a_n, \hspace{1cm}\forall n\geq
0, \label{Eq_n-coef}
\end{equation}
and in particular $\left\langle t^k|x^n\right\rangle
=\cnq{n}\delta_{n,k}$, for all $n,k\geq 0$, where $\delta_{n,k}$ is the
Kronecker delta function.

Let $f_L(t)=\sum_{k\geq 0}\frac{\lla L|x^k\rra
}{\cnq{k}} t^k$, then we get $\left\langle f_L(t)|x^n\right\rangle
=\left\langle L|x^n\right\rangle$. Thus the map $L\mapsto F_L(t)$
is a vector space isomorphism from $P^*$ onto $\mathcal{F}$.
Henceforth, $\mathcal{F}$ denotes both the algebra of formal power
series in $t$ and the vector space of all linear functionals in
$P$ (see \cite{Roman1982}), so that $\mathcal{F}$ is an umbral algebra, 
and the umbral calculus is the study of the umbral calculus. Note that
the umbral calculus considered here is non-classical because it is 
associated with the sequence $\cnq{n}$ instead of classical $n!$. From 
\eqref{Def::DegExpSer}-\eqref{Eq_n-coef} one can easily see that 
$\lla\eqT(yt)|x^n\rra=y^n$ and, respectively, $\lla\eqT(yt)|(p(x)\rra=p(y)$.
For all $f(t)\in\mathcal{F}$ and for all polynomials $p(x)\in P$
we have
\begin{align*}
f(t)=\sum\limits_{k\geq 0} \frac{\left\langle f(t)|x^k\right\rangle }{c_{k;q}}
t^k\mbox{ and }
p(x)=\sum\limits_{k\geq 0} \frac{\left\langle
t^k|p(x)\right\rangle }{c_{k;q}}x^k.
\end{align*}
For $f_1(t),\ldots,f_m(t)\in\mathcal{F}$, we have  (see
\cite{Roman1982, Roman1984})
$$\left\langle f_1(t)\cdots f_m(t)|x^n\right\rangle
=\sum\limits_{\substack{i_1+\cdots+i_m=n \\ i_j\geq 0}}
\frac{c_{n;q}}{c_{i_1;q}\cdots c_{i_m;q}} \left\langle f_1(t)|
x^{i_1}\right\rangle
\cdots \left\langle f_m(t)|x^{i_m}\right\rangle.$$ 
Let us define a linear operator $\Dcq{t}$ as following
\begin{align*}
\Dcq{t} t^n=
\left\{\begin{array}{ll}
\frac{\cnq{n}}{\cnq{n-1}}t^{n-1},\quad &\text{ for integer }
n\geq1,\\
0,\quad &\text{ }n=0.\end{array}\right.
\end{align*}
Therefore for any polynomial $p(x)=\sum_{j=0}^n a_jx^j$
we have
$\lla t^k|p(x)\rra=\cnq{k}a_k=\Dcqk{x}{k}p(0)$,
and, in particular, $\lla t^0|p(x)\rra=p(0)$. 

For any $f(t)\in\mathcal{F}$ the linear operator $f(t)$ on $\mathcal{F}$ is 
defined by (see \cite{Roman1982})
$f(t)x^n=\sum_{k=0}^n\frac{\cnq{n}}{\cnq{k}\cnq{n-k}}a_kx^{n-k}$,
which leads to
\begin{align}
t^kx^n=
\left\{\begin{array}{ll}
\frac{\cnq{n}}{\cnq{n-k}}x^{n-k},\quad &\text{ for integer }
n\geq k,\\
0,\quad &\text{ }n<k.\end{array}\right.\label{DefLinOper-tk}
\end{align}

For $f(t),g(t)\in \mathcal{F}$, it holds that
$\lla f(t)g(t)|p(x)\rra=\lla g(t)|f(t)p(x)\rra=\lla f(t)|g(t)p(x)\rra$.
The {\em degree} of $f(t)$ (denoted by $o(f(t))$) is the smallest $k$ 
such that $t^k$
does not vanish. If $o(f(t))=0$ then the series $f(t)$ is called 
{\em invertible} and has a multiplicative inverse denoted by  $f^{-1}(t)$ 
or $1/f(t)$. If $o(f(t))=1$ then the series $f(t)$ is called {\em delta series} 
and has a compositional inverse $\bar{f}(t)$ satisfying
$f(\bar{f}(t))=\bar{f}(f(t))=t$.
For a delta series $f(t)\in\mathcal{F}$ and an invertible series 
$g(t)\in\mathcal{F}$ we say that a polynomial sequence $s_n(x)$  is a 
{\em Sheffer sequence} for the pair $(g(t),f(t))$ and denote it by 
$s_n(x)\sim (g(t),f(t))$ if for all $n,k\geq0$ it holds that
$\lla g(t)f(t)^k|s_n(x)\rra=\cnq{n}\delta_{n,k}$.
Thus,  $s_n(x)\sim(g(t),f(t))$ if and only if
\begin{align}
\frac{1}{g(\bar{f}(t))}\eqT(y\bar{f}(t))=\sum_{n=0}^\infty\frac{s_n(y)}
{\cnq{n}}t^n,\label{ShefSeqDef1}
\end{align}
for all $y\in\mathbb{C}$ (see \cite{Roman1984}).
The following statements are equivalent
\begin{align}
\begin{array}{l}
s_n(x)\sim(g(t),f(t)),\\
g(t)s_n(x)\sim(1,f(t)),\\
f(t)s_n(x)=\frac{\cnq{n}}{\cnq{n-1}}s_{n-1}(x).
\end{array}
\label{p5eq2}
\end{align}
Moreover, the following theorem holds.
\begin{theorem}\label{thrm::PolExp}
Let $s_n(x)\sim(g(t),f(t))$. Then 
\begin{itemize}
\item[(i)]  for any polynomial $p(x)$, $p(x)=\sum_{n= 0}^\infty \frac{\lla 
g(t)f(t)^n|p(x)\rra}{\cnq{n}}s_n(x)$ (Polynomial Expansion, Theorem~6.2.3 
in \cite{Roman1984});
\item[(ii)] $s_n(x)=\sum_{k=0}^n\frac{\lla g(\bar{f}(t))^{-1}\bar{f}(t)^k|
x^n\rra}{\cnq{k}}x^k$ (Conjugate Representation, Theorem~6.2.5 in 
\cite{Roman1984});
\end{itemize}
Moreover, a sequence $s_n(x)\sim(g(t),f(t))$ for some invertible $g(t)$, if 
and only if
$$\eqT(yt)s_n(x)=\sum_{k=0}^n\frac{\cnq{n}}{\cnq{k}\cnq{n-k}}p_k(y)
s_{n-k}(x)$$
for all constants $y$, where $p(x)\sim(1,f(t))$ (The Sheffer Identity, 
Theorem~6.2.8 in \cite{Roman1984}).
\end{theorem}

Moreover, if $s_n(x)\sim(g(t),t)$ then,
from \eqref{ShefSeqDef1} we obtain
\begin{align}
\frac{1}{g(t)}x^n=s_n(x), \quad\text{or} \quad x^n=g(t)s_n(x).\label{p5eq1}
\end{align}

Let us define now an inverse operator for operator $\Dcq{x}$ as
following
\begin{align}
\Icq{x}x^n=\intcq{x^n}{x}=\frac{\cnq{n}}{\cnq{n+1}}x^{n+1}.\label{p5eq5}
\end{align}
Obviously, $\Icq{x}(\Dcq{x}x^n)=\Dcq{x}(\Icq{x}x^n)=x^n$.

\section{Bernoulli polynomials}
Let us define {\em degenerate Bernoulli polynomials $\dbp{n}{x}$ and 
numbers $\dbn{n}$} as
\begin{align}
\frac{t}{\eqT(t)-1}\,\eqT(xt)&=\sum_{n=0}^\infty\dbp{n}{x}\frac{t^n}{c_{n;q}}. 
\label{GF::DBP}\\
\frac{t}{\eqT(t)-1}&=\sum_{n=0}^\infty\dbn{n}\frac{t^n}{c_{n;q}}. 
\label{GF::DBN}
\end{align}
The first few values of degenerate Bernoulli polynomials and numbers are 
listed in Table~\ref{Table1}.
\begin{table}[!h]
\begin{tabular}{cll}
\hline
$n$ & \quad$\dbp{n}{x}$ & $\dbn{n}$ \\
\hline
$0$ & \quad$1$ & \quad$1$\\[4pt]
$1$ & \quad$x-\frac{q}{2}$ & \quad$-\frac{q}{2}$ \\[4pt]
$2$ & \quad$x^2-x+\frac{1}{3}-\frac{q}{6}$ & \quad$\frac{1}{3}-
\frac{q}{6}$\\[4pt]
$3$ & \quad$\frac{-4(2q-1)x^3+6qx^2+2(q-2)x+q^2-3q+2}{4(2q-1)}$ & 
\quad$\frac{q^2-3q+2}{4(2q-1)}$\\[4pt]
$4$ & \quad$\frac{(6q^2-7q+2)x^4-2(2q^2-q)x^3-(q^2-2)x^2-(q^2-3q+2)x-
\frac{1}{30}(19q^3-76q^2+94q-36)}{(2q-1)(3q-2)}$ & \quad$-
\frac{19q^3-76q^2+94q-36}{30(2q-1)(3q-2)}$\\[4pt]
\hline
\end{tabular}
\caption{Degenerate Bernoulli polynomials and numbers.}\label{Table1}
\end{table}
Let us denote by $\dbu$ the umbra of the Bernoulli numbers sequence, 
that is,
\begin{align*}
\frac{t}{\eqT(t)-1}=\eqT(\dbu t)\implies t=\eqT(\dbu t)\eqT(t)-\eqT(\dbu t).
\end{align*}
By applying \eqref{p5eq10}, we have 
$t=\sum_{n=0}^\infty\frac{(\dbu+1)^n_ct^n}{\cnq{n}}
-\sum_{n=0}^\infty\frac{\dbn{n}t^n}{\cnq{n}}$. 
Thus, $(\dbu+1)^n_c-\dbn{n}=\delta_{1,n}$,
which is a generalization of a well-known identity for the classical 
Bernoulli numbers.
From \eqref{GF::DBP} we obtain,
\begin{align*}
\sum_{n=0}^\infty\dbp{n}{x}\frac{t^n}{c_{n;q}}&=
\frac{t}{\eqT(t)-1}\,\eqT(xt)=\sum_{n= 0}^\infty\dbn{n}\frac{t^n}{c_{n;q}}
\cdot
\sum_{k= 0}^\infty\frac{t^kx^k}{c_{k;q}}\\
&=\sum_{n=0}^\infty\sum_{k=0}^n\binomcq{n}{k}\dbn{k}x^{n-k}
\frac{t^n}{c_{n;q}},
\end{align*}
and, by extracting the coefficients of $\dfrac{t^n}{c_{n;q}}$, we get an
analogue of the well known identity for the Bernoulli polynomials.
\begin{proposition}\label{p5prop1}
For all $n\in\mathbb{N}$, the degenerated Bernoulli polynomials 
$\dbp{n}{x}$ defined by \eqref{GF::DBP} satisfy
\begin{align*}
\dbp{n}{x}=\sum_{k=0}^n\binomcq{n}{k}\dbn{n-k}x^k.
\end{align*}
Moreover, for all $n\geq 0$ and $x\in\mathbb{C}$, it holds that
$\dbp{n}{x}=(\dbu+x)^n_c$ and $\dbp{n}{1}-\dbn{n}=\delta_{1,n}$.
\end{proposition}
Note that the identity of the Proposition~\ref{p5prop1}
can be obtained immediately from noticing that 
$\dbp{n}{x}\sim(\frac{\eqT(t)-1}{t},t)$ and applying Theorem 
\ref{thrm::PolExp}(ii). From \eqref{p5eq1}, we obtain
\begin{align}
\frac{t}{\eqT(t)-1}x^k=\dbp{k}{x},\quad\text{or}\quad x^k=\frac{\eqT(t)-1}{t}
\dbp{k}{x}.\label{p5eq4}
\end{align}
By applying \eqref{p5eq2}, we get
$t\dbp{n}{x}=\frac{\cnq{n}}{\cnq{n-1}}\dbp{n-1}{x}$.

\begin{lemma}
For any polynomial $p(x)=\sum\limits_{k=0}^na_kx^k\in P$, it holds that
$$\lla\frac{\eqT(yt)-1}{t}|p(x)\rra=\int_0^yp(u)\dcq u.$$
\end{lemma}
\begin{proof}
Let us consider the action of this linear functional on a monomial $x^j$. 
From \eqref{Def::DegExpSer}, we obtain
\begin{align}
\lla\frac{\eqT(yt)-1}{t}\Big|x^j\rra=\lla \eqT(yt)-1 \Big|\frac{1}{t}x^j\rra.
\label{p5eq19a}
\end{align}
The operator $\frac{1}{t}$ is the inverse of the operator $t$ defined as 
$tx^j=\frac{\cnq{j}}{\cnq{j-1}}x^{j-1}$. Therefore, by applying $\frac{1}{t}$ 
to both sides of this equation, we obtain $x^j=\frac{1}{t}\frac{\cnq{j}}
{\cnq{j-1}}x^{j-1}$ and, thus, $\frac{1}{t}x^{j-1}=\frac{\cnq{j-1}}{\cnq{j}}x^j$.
So by \eqref{p5eq19a}, we have 
\begin{align*}
\lla\frac{\eqT(yt)-1}{t}\Big|x^j\rra&=\lla \eqT(yt)-1\Big| \frac{\cnq{j}}
{\cnq{j+1}}x^{j+1}\rra
=\frac{\cnq{j}}{\cnq{j+1}}y^{j+1}=\int_0^yx^{j}\dcq x.
\end{align*}
By linearity, we complete the proof.
\end{proof}

By applying this Lemma to the polynomials $\dbp{n}{x}$ and using 
\eqref{p5eq4},
we obtain
\begin{align}
\int_0^1\dbp{n}{u}\dcq{u}&=\lla\frac{\eqT(t)-1}{t}\Big|\dbp{n}{x}\rra
=\lla1 \Big| \frac{\eqT(t)-1}{t}\dbp{n}{x}\rra\nn\\
&=\lla 1|x^n\rra=\lla t^0 | x^n\rra=\cnq{n}\delta_{n,0}.\label{p5eq6}
\end{align}
From another side, by Proposition~\ref{p5prop1} we have
\begin{align}
\int_0^1\dbp{n}{u}\dcq{u} &=\int_0^1\sum_{k=0}^n\binomcq{n}{k}\dbn{n-k}
u^k\dcq{u}
=\sum_{k=0}^n\binomcq{n}{k}\dbn{n-k}\int_0^1u^k\dcq{u},
\end{align}
and, by using the definition \eqref{p5eq5}, we obtain
\begin{align}
\int_0^1\dbp{n}{u}\dcq{u}
&=\sum_{k=0}^n\binomcq{n}{k}\dbn{n-k}\frac{\cnq{k}}{\cnq{k+1}}
u^{k+1}\Big|_0^1
=\sum_{k=0}^n\frac{\cnq{n}}{\cnq{k+1}\cnq{n-k}}\dbn{n-k}\nn\\
&=\frac{\cnq{n}}{\cnq{n+1}}\sum_{k=0}^n\binomcq{n+1}{k}\dbn{k}. 
\label{p5eq7}
\end{align}
Hence, by comparing \eqref{p5eq6} with \eqref{p5eq7} and bringing into 
consideration that $\cnq{1}=1$, we can state the following result.
\begin{proposition}
For all integer $n\geq 0$, it holds that 
$\sum_{k=0}^n\binomcq{n+1}{k}\dbn{k}=\delta_{n,0}$.
\end{proposition}

\begin{remark}
This Proposition brings another formulation and proof of the Corollary of 
the Proposition~\ref{p5prop1} at $x=1$.
\end{remark}

One of the very important aspects in the theory of orthogonal polynomials 
is a connection between different kinds of polynomials.
Let us consider the equation \eqref{p5eq4}. It can be rewritten as
\begin{align}
tx^k=(\eqT(t)-1)\dbp{k}{x}.
\end{align}
Therefore, by \eqref{DefLinOper-tk}, we obtain
\begin{align}
\frac{\cnq{k}}{\cnq{k-1}}x^{k-1}&=\eqT(t)\dbp{k}{x}-\dbp{k}{x}
=\sum_{j=0}^k\binomcq{k}{j}\dbp{k}{x}-\dbp{k}{x}
=\sum_{j=0}^{k-1}\binomcq{k}{j}\dbp{k}{x}.\nn
\end{align}
Thus we can state the following result.
\begin{proposition}\label{p5::prop1}
For all integer $n\geq 0$ it holds that
$x^n=\frac{\cnq{n}}{\cnq{n+1}}\sum_{k=0}^n\binomcq{n+1}{k}\dbp{k}{x}$. 
Moreover, 
for all integer $n\geq 0$, $[x^n]\dbp{n}{x}=1$.
\end{proposition}

Now, we are ready to present an analogue of the Euler identity for 
Bernoulli polynomials and numbers.

\begin{theorem} For all integer $n\geq 2$, the degenerate Bernoulli 
polynomials defined by \eqref{GF::DBP} satisfy
\begin{align}
\sum_{k=0}^n\binomcq{n}{k}\dbp{k}{x}\dbp{n-k}{y}&=-(n-1)\dbp{n}
{(x+y)_c}-n\dbp{n-1}{(x+y)_c}\frac{(n-1)-q(n-2)}{(n-1)q-(n-2)}\nn\\
&+\widehat{\dbn{n}}((x+y)_c)+(1-q)\frac{\cnq{n}}{\cnq{n-1}}
\widehat{\dbn{n-1}}((x+y)_c),
\end{align}
where $\widehat{\dbn{n}}(u)=\sum_{k=0}^n\binomcq{n}{k}k\dbn{n-k}u^k$.
\end{theorem}
\begin{proof}
Let $b(t)=\frac{t}{\eqT(t)-1}$. Therefore
\begin{align}
b(t)^2=(1-qt)b(t)-(1+(1-q)t)tb'(t),
\end{align}
where $b'(t)=\frac{d}{dt}b(t)$. By multiplying both sides by $\eqT(xt)
\eqT(yt)$ and replacing $b'(t)\eqT(ut)=(b(t)\eqT(ut))'-b(t)\eqT'(ut)$ in 
accordance with Leibniz rule, 
we obtain
\begin{align}
b(t)^2\eqT(xt)\eqT(yt)&=(1-qt)b(t)\eqT((x+y)_ct)\nn\\
&-(t+(1-q)t^2)\left[\left(b(t)\eqT((x+y)_ct)\right)'-b(t)\eqT'((x+y)_ct)\right].
\end{align}
It follows that
\begin{align}
\sum_{n=0}^\infty\sum_{k=0}^n\binomcq{n}{k}&\dbp{k}{x}
\dbp{n-k}{y}\frac{t^n}{\cnq{n}}=
(1-qt)\sum_{n=0}^\infty\dbp{n}{(x+y)_c}\frac{t^n}{\cnq{n}}\nn\\
&-(t+(1-q)t^2)\left(\sum_{n=0}^\infty\dbp{n}{(x+y)_c}\frac{t^n}{\cnq{n}}
\right)'\nn\\
&+(t+(1-q)t^2)\sum_{n=0}^\infty\dbn{n}\frac{t^n}{\cnq{n}}\cdot\left(
\sum_{k=0}^\infty(x+y)^k_c\frac{t^k}{\cnq{k}}\right)'.
\end{align}
After differentiating and applying the Cauchy product, one can extract the 
coefficients of $\frac{t^n}{\cnq{n}}$ for $n\geq 2$ as follows.
\begin{align}
\sum_{k=0}^n\binomcq{n}{k}\dbp{k}{x}
\dbp{n-k}{y}&=\dbp{n}{(x+y)_c}-\frac{\cnq{n}}{\cnq{n-1}}q\dbp{n-1}
{(x+y)_c}\nn\\
&-n\dbp{n}{(x+y)_c}-\frac{\cnq{n}}{\cnq{n-1}}(1-q)(n-1)\dbp{n-1}{(x+y)_c}
\nn\\
&+\sum_{k=0}^{n-1}\binomcq{n}{k}(n-k)\dbn{k}\cdot(x+y)_c^{n-k}\nn\\
&+\sum_{k=0}^{n-2}\binomcq{n-1}{k}(n-k-1)(1-q)\dbn{k}\cdot
(x+y)^{n-k-1}_c\frac{\cnq{n}}{\cnq{n-1}}.
\end{align}
Rearranging the summation indexes, denoting 
$\widehat{\dbn{n}}(u)=\sum_{k=0}^n\binomcq{n}{k}k\dbn{n-k}u^k$, and 
gathering the similar terms complete the proof.
\end{proof}
An analogue of the Euler identity for the degenerate Bernoulli numbers
follows immediately from the previous Theorem by assuming $x=y=0$ 
and noting that $\widehat{\dbn{n}}(0)=0$ for all integer $n\geq 0$.
\begin{theorem}
For all integer $n\geq 2$, the degenerate Bernoulli numbers defined by 
\eqref{GF::DBN} satisfy
$$\sum_{k=1}^n\binomcq{n}{k}\dbn{k}\dbn{n-k}=-n\dbn{n}-
n\dbn{n-1}\frac{(n-1)-q(n-2)}{(n-1)q-(n-2)}.$$
\end{theorem}

\section{Euler polynomials}
Let us define {\em degenerate Euler polynomials $\dep{n}{x}$ and values 
$\den{n}=\dep{n}{0}$} as
\begin{align}
\frac{2}{\eqT(t)+1}\eqT(xt)&=\sum_{n=0}^\infty\dep{n}{x}\frac{t^n}{\cnq{n}}, 
\label{GF::DEP}\\
\frac{2}{\eqT(t)+1}&=\sum_{n=0}^\infty\den{n}\frac{t^n}{\cnq{n}}.
\label{GF::DEN}
\end{align}
The first five degenerate Euler polynomials and their special values are 
listed in Table~\ref{Table2}.
\begin{table}[!h]
\begin{tabular}{cll}
\hline
$n$ & \quad$\dep{n}{x}$ & $\den{n}$ \\
\hline
$0$ & \quad$1$ & \quad$1$\\[4pt]
$1$ & \quad$x-\frac{1}{2}$ & \quad$-\frac{1}{2}$ \\[4pt]
$2$ & \quad$\frac{2x^2q-2x-q+1}{2q}$ & \quad$\frac{1-q}{2q}$\\[4pt]
$3$ & \quad$\frac{(8q^2-4q)x^3-6qx^2+(6-6q)x-4q^2+8q-3}{4q(2q-1)}$ & 
\quad$\frac{-4q^2+8q-3}{4q(2q-1)}$\\[4pt]
$4$ & \quad$\frac{(12q^3-14q^2+4q)x^4-(8q^2-4q)x^3-(6q^2-6q)x^2-
(8q^2-16q+6)x-6q^3+18q^2-15q+3}{2q(2q-1)(3q-2)}$ & \quad$
\frac{-6q^3+18q^2-15q+3}{2q(2q-1)(3q-2)}$\\[4pt]
\hline
\end{tabular}
\caption{Degenerate Euler polynomials and values.}\label{Table2}
\end{table}
It is easy to see that $\dep{n}{x}\sim\left(\frac{\eqT(t)+1}{2},t\right)$.
Therefore, by \eqref{p5eq2}, we obtain
$$t\dep{n}{x}=\frac{\cnq{n}}{\cnq{n-1}}\dep{n-1}{x}\mbox{ and }
\Dcq{x}^k\dep{n}{x}=\frac{\cnq{n}}{\cnq{n-k}}\dep{n-k}{x}.$$
From \eqref{p5eq1}, we have 
\begin{align}
\frac{2}{\eqT(t)+1}x^k=\dep{k}{x},\mbox{ or } x^k=
\frac{\eqT(t)+1}{2}\dep{k}{x}.\label{p5eq9}
\end{align}
Let us denote by $\deu$ the umbra of the Euler values sequence, that is,
\begin{align*}
\frac{2}{\eqT(t)+1}=\eqT(\deu t)\implies 2=\eqT(\deu t)\eqT(t)+\eqT(\deu t).
\end{align*}
By applying \eqref{p5eq10}, we obtain
$2=\sum_{n=0}^\infty\frac{(\deu+1)^n_ct^n}{\cnq{n}}+
\sum_{n=0}^\infty\frac{\den{n}t^n}{\cnq{n}}$.
Thus,  
$(\deu+1)^n_c+\den{n}=2\delta_{0,n}$,
which is a generalization of a well-known identity for the classical case. 
From \eqref{GF::DEP}-\eqref{GF::DEN}, we get
\begin{align}
\sum_{n=0}^\infty \dep{n}{x}\frac{t^n}{\cnq{n}}&=
\frac{2}{\eqT(t)+1}\cdot\eqT(xt)
=\sum_{n=0}^\infty\den{n}\frac{t^n}{\cnq{n}}\cdot\sum_{k=0}^\infty 
\frac{x^kt^k}{\cnq{k}}\nn\\
&=\sum_{n=0}^\infty\sum_{k=0}^n\frac{\cnq{n}}{\cnq{k}\cnq{n-k}}\den{n-k}
x^k\frac{t^n}{\cnq{n}},\nn
\end{align}
which leads to the following proposition.
\begin{proposition}\label{p5Prop::prop2}
For all $n\in\mathbb{N}$, the degenerated Euler polynomials $\dep{n}{x}$ 
defined by \eqref{GF::DEP} satisfy
\begin{align*}
\dep{n}{x}=\sum_{k=0}^n\binomcq{n}{k}\den{n-k}x^k.
\end{align*}
Moreover, for all $n\geq 0$ and $x\in\mathbb{C}$, 
$\dep{n}{x}=(\deu+x)^n_c$ and $\dep{n}{1}+\den{n}=2\delta_{0,n}$.
\end{proposition}

From the Theorem~\ref{thrm::PolExp}, by assuming $y=1$, we obtain
$\eqT(t)\dep{n}{x}=\sum_{k=0}^n\frac{\cnq{n}}{\cnq{k}\cnq{n-k}}
\dep{k}{x}$.
Applying this identity to \eqref{p5eq9} leads to the next result.
\begin{proposition}\label{Prop::prop1}
For all integer $n\geq 0$, it holds that
\begin{align*}
x^n=\frac{1}{2}\sum_{k=0}^n\binomcq{n}{k}\dep{k}{x}+\frac{1}{2}
\dep{n}{x}.
\end{align*}
Moreover, for all $n\geq 0$, $[x^n]\dep{n}{x}=1$.
\end{proposition}
By substituting $x=0$ into the statement of the 
Proposition~\eqref{Prop::prop1} and rearranging the terms, we obtain 
the following result.

\begin{corollary}
For all integer $n\geq 1$, it holds that
$-2\den{n}=\sum_{k=0}^{n-1}\binomcq{n}{k}\den{k}$.
\end{corollary}

\section{Genocchi polynomials}
Let us define {\em degenerate Genocchi polynomials $\dgp{n}{x}$ and 
numbers $\dgn{n}$} as
\begin{align}
\frac{2t}{\eqT(t)+1}\,\eqT(xt)&=\sum_{n=0}^\infty\dgp{n}{x}
\frac{t^n}{c_{n;q}}, \label{GF::DGP}\\
\frac{2t}{\eqT(t)+1}&=\sum_{n=0}^\infty\dgn{n}\frac{t^n}{c_{n;q}}.
\label{GF::DGN}
\end{align}
The first few values of degenerate Genocchi polynomials and numbers 
are listed in Table~\ref{Table3}.
\begin{table}[!h]
\begin{tabular}{cll}
\hline
$n$ & \quad$\dgp{n}{x}$ & $\dgn{n}$ \\
\hline
$0$ & \quad$0$ & \quad$0$\\[4pt]
$1$ & \quad$1$ & \quad$1$ \\[4pt]
$2$ & \quad$\frac{2x-1}{q}$ & \quad$-\frac{1}{q}$\\[4pt]
$3$ & \quad$\frac{6qx^2-6x-3q+3}{2q(2q-1)}$ & \quad
$\frac{3-3q}{2q(2q-1)}$\\[4pt]
$4$ & \quad$\frac{(8q^2-4q)x^3-6qx^2+(6-6q)x-(4q^2-8q+3)}{q(2q-1)
(3q-2)}$ & \quad$-\frac{4q^2-8q+3}{q(2q-1)(3q-2)}$\\[4pt]
\hline
\end{tabular}
\caption{Degenerate Genocchi polynomials and numbers.}\label{Table3}
\end{table}
It is easy to see that $\dgp{n}{x}\sim\left(\frac{\eqT(t)+1}{2t},t\right)$.
Moreover, by comparing \eqref{GF::DEP} with \eqref{GF::DGP}, one can 
immediately conclude that $deg(\dgp{n}{x})=n-1$. Therefore, by 
\eqref{p5eq2}, we obtain
$$t\dgp{n}{x}=\frac{\cnq{n}}{\cnq{n-1}}\dgp{n-1}{x}\mbox{ and }
\Dcq{x}^k\dgp{n}{x}=\frac{\cnq{n}}{\cnq{n-k}}\dgp{n-k}{x}.$$
From \eqref{p5eq1}, we have
\begin{align}
\frac{2t}{\eqT(t)+1}x^k=\dgp{k}{x}\text{ or } x^k=\frac{\eqT(t)+1}{2t}
\dgp{k}{x}.\label{p5eq17}
\end{align}
Let us denote by $\dgu$ the umbra of the Genocchi numbers sequence, 
that is,
$$\frac{2t}{\eqT(t)+1}=\eqT(\dgu t)\implies 2t=\eqT(\dgu t)\eqT(t)+
\eqT(\dgu t).$$
By applying \eqref{p5eq10}, we get that
$2t=\sum_{n=0}^\infty\frac{(\dgu+1)^n_ct^n}{\cnq{n}}+\sum_{n=0}^\infty
\frac{\dgn{n}t^n}{\cnq{n}}$.
So, $(\dgu+1)^n_c+\dgn{n}=2\delta_{1,n}$,
which is a generalization of a well-known identity for the classical case.

From the definitions of degenerate Genocchi numbers and polynomials 
\eqref{GF::DGP}-\eqref{GF::DGN}, we obtain
\begin{align*}
\sum_{n=0}^\infty \dgp{n}{x}\frac{t^n}{\cnq{n}}&=\frac{2t}{\eqT(t)+1}
\cdot\eqT(xt)
=\sum_{n=0}^\infty\dgn{n}\frac{t^n}{\cnq{n}}\cdot\sum_{k=0}^\infty 
\frac{x^kt^k}{\cnq{k}}\\
&=\sum_{n=0}^\infty\sum_{k=0}^n\frac{\cnq{n}}{\cnq{k}\cnq{n-k}}\dgn{n-k}
x^k\frac{t^n}{\cnq{n}},
\end{align*}
and we can state the following proposition.
\begin{proposition}\label{p5Prop::prop4}
For all $n\in\mathbb{N}$, the degenerated Genocchi polynomials $
\dgp{n}{x}$ defined by \eqref{GF::DGP} satisfy
$\dgp{n}{x}=\sum_{k=0}^n\binomcq{n}{k}\dgn{n-k}x^k$.
Moreover, for all $n\in\mathbb{N}$ and $x\in\mathbb{C}$, 
$\dgp{n}{x}=(\dgu+x)^n_c$ and  $\dgp{n}{1}+\dgn{n}=2\delta_{1,n}$.
\end{proposition}

From the Theorem~\ref{thrm::PolExp} with $y=1$, we obtain
$\eqT(t)\dgp{n}{x}=\sum_{k=0}^n\frac{\cnq{n}}{\cnq{k}\cnq{n-k}}
\dgp{k}{x}$.
Applying this identity to the equation~\eqref{p5eq17} leads to the next 
result.
\begin{proposition}\label{Prop::prop5}
For all integer $n\geq 0$, it holds that
\begin{align*}
x^n=\frac{1}{2}\frac{\cnq{n}}{\cnq{n+1}}\sum_{k=0}^{n+1}\binomcq{n+1}{k}
\dgp{k}{x}+\frac{1}{2}\frac{\cnq{n}}{\cnq{n+1}}\dgp{n+1}{x}.
\end{align*}
Moreover, for all $n\geq 0$,  $[x^n]\dgp{n+1}{x}=
\frac{\cnq{n+1}}{\cnq{n}}$.
\end{proposition}

By substituting $x=0$ into the statement of the 
Proposition~\eqref{Prop::prop1} and rearranging the terms, we have 
the following corollary.

\begin{corollary}
For all integer $n\geq 1$, it holds that
$-2\dgn{n+1}=\sum_{k=0}^{n}\binomcq{n+1}{k}\dgn{k}$.
\end{corollary}

\section{Connections between polynomials}
We already have shown a connection between monomials $p(x)=x^n$ 
and degenerated Bernoulli $\dbp{n}{x}$, degenerated Euler $\dep{n}{x}$, 
and degenerated Genocchi $\dgp{n}{x}$ polynomials.
Let us assume now that a polynomial $p(x)\in P$ of degree $n$ can be 
expressed as a linear combination of the deformed Bernoulli polynomials 
$p(x)=\sum_{k=0}^nb_k\dbp{k}{x}$. Therefore, by Theorem 
\ref{thrm::PolExp}(i), we obtain
\begin{align*}
\sum_{k=0}^nb_k\dbp{k}{x}=\sum_{k=0}^n\frac{\lla \frac{\eqT(t)-1}{t}\cdot 
t^k \Big| p(x)\rra}{\cnq{k}}\dbp{k}{x},
\end{align*}
where
\begin{align*}
b_k&=\frac{1}{\cnq{k}}\lla \frac{\eqT(t)-1}{t}\cdot t^k \Big| p(x)\rra
=\frac{1}{\cnq{k}}\lla \frac{\eqT(t)-1}{t}\cdot\Big|  t^k p(x)\rra\nn\\
&=\frac{1}{\cnq{k}}\lla \frac{\eqT(t)-1}{t}\cdot\Big|  \Dcq{x}^k p(x)
\rra=\frac{1}{\cnq{k}}\int_0^1 \Dcq{x}^k p(x)\dcq{x}.
\end{align*}
Thus, we can state the following statement.
\begin{proposition}\label{p5Prop::prop2}
For any polynomial $p(x)\in P$ of degree $n$, there exist constants 
$b_0,b_1,\ldots,b_n$ such that
$p(x)=\sum_{k=0}^nb_k\dbp{k}{x}$,
where $b_k=\frac{1}{\cnq{k}}\int_0^1 \Dcq{x}^{k} p(x)\dcq{x}$.
\end{proposition}

\begin{theorem}
Let us define $\binomcq{n+1}{k,m,n-k-m+1}=\frac{\cnq{n+1}}{\cnq{k}
\cnq{m}\cnq{n-k+1-m}}$. Then, for all integer $n\geq 0$,
\begin{align*}
\dep{n}{x}=\frac{\cnq{n}}{\cnq{n+1}}\sum_{k=0}^n\sum_{m=0}^{n-k}
\binomcq{n+1}{k,m,n-k-m+1}\den{m}\dbp{k}{x},
\end{align*}
or
\begin{align*}
\dep{n}{x}=-2\frac{\cnq{n}}{\cnq{n+1}}\sum_{k=0}^n\binomcq{n+1}{k}
\den{n-k+1}\dbp{k}{x}.
\end{align*}
\end{theorem}
\begin{proof}
Let us assume that $\dep{n}{x}=\sum_{k=0}^nb_k\dbp{k}{x}$. Therefore,
by Proposition~\eqref{p5Prop::prop2}, we have
\begin{align*}
b_k&=\frac{1}{\cnq{k}}\int_0^1\Dcq{x}^k\dep{n}{x}\dcq x
=\frac{1}{\cnq{k}}\int_0^1\frac{\cnq{n}}{\cnq{n-k}}\dep{n-k}{x}\dcq x\\
&=\frac{\cnq{n}}{\cnq{k}\cnq{n-k}}\cdot\frac{\cnq{n-k}}{\cnq{n-k+1}}
\dep{n-k+1}{x}\Big|_0^1
=\frac{\cnq{n}}{\cnq{k}\cnq{n-k+1}}\left(\dep{n-k+1}{1}-\den{n-k+1}\right).
\end{align*}
In accordance with Proposition~\eqref{p5Prop::prop2} for $x=1$, we 
obtain
\begin{align*}
b_k&=\frac{\cnq{n}}{\cnq{k}\cnq{n-k+1}}\left(\sum_{m=0}^{n-k+1}
\frac{\cnq{n-k+1}}{\cnq{m}\cnq{n-k+1-m}}\den{m}-\den{n-k+1}\right)\\
&=\frac{\cnq{n}}{\cnq{k}\cnq{n-k+1}}\sum_{m=0}^{n-k}
\frac{\cnq{n-k+1}}{\cnq{m}\cnq{n-k+1-m}}\den{m}
=\sum_{m=0}^{n-k}\frac{\cnq{n}}{\cnq{k}\cnq{m}\cnq{n-k+1-m}}\den{m}\\
&=\frac{\cnq{n}}{\cnq{n+1}}\sum_{m=0}^{n-k}\frac{\cnq{n+1}}{\cnq{k}
\cnq{m}\cnq{n-k+1-m}}\den{m}.
\end{align*}
On the other side, by Proposition \ref{p5Prop::prop2}, we have 
\begin{align*}
b_k&=\frac{\cnq{n}}{\cnq{k}\cnq{n-k+1}}\left(2\delta_{0,n-k+1}-2
\den{n-k+1}\right).
\end{align*}
Therefore,
\begin{align*}
\dep{n}{x}&=2\frac{\cnq{n}}{\cnq{n+1}}\sum_{k=0}^n\binomcq{n+1}{k}
\left(\delta_{0,n-k+1}-\den{n-k+1}\right)\dbp{k}{x}\nn\\
&=-2\frac{\cnq{n}}{\cnq{n+1}}\sum_{k=0}^n\binomcq{n+1}{k}
\den{n-k+1}\dbp{k}{x},
\end{align*}
which completes the proof.
\end{proof}

\begin{theorem}
For all integer $n\geq 1$,
\begin{align*}
\dgp{n}{x}=-2\frac{\cnq{n}}{\cnq{n+1}}\sum_{k=0}^{n-1}\binomcq{n+1}{k}
\dgn{n-k+1}\dbp{k}{x},
\end{align*}
or
\begin{align*}
\dgp{n}{x}=\frac{\cnq{n}}{\cnq{n+1}}\sum_{k=0}^n\sum_{m=0}^{n-k}
\binomcq{n+1}{k,m,n+1-k-m}\dgn{m}\dbp{k}{x}.
\end{align*}
\end{theorem}
\begin{proof}
Let us assume that $\dgp{n}{x}=\sum_{k=0}^nb_k\dbp{k}{x}$. Therefore,
by Proposition~\eqref{p5Prop::prop2}, we have
\begin{align*}
b_k&=\frac{1}{\cnq{k}}\int_0^1\Dcq{x}^k\dgp{n}{x}\dcq x
=\frac{1}{\cnq{k}}\int_0^1\frac{\cnq{n}}{\cnq{n-k}}\dgp{n-k}{x}\dcq x\\
&=\frac{\cnq{n}}{\cnq{k}\cnq{n-k}}\cdot\frac{\cnq{n-k}}{\cnq{n-k+1}}
\dgp{n-k+1}{x}\Big|_0^1
=\frac{\cnq{n}}{\cnq{k}\cnq{n-k+1}}\left(\dgp{n-k+1}{1}-\dgn{n-k+1}\right).
\end{align*}
By Proposition \ref{p5Prop::prop2}, we obtain
$b_k=\frac{\cnq{n}}{\cnq{k}\cnq{n-k+1}}\left(2\delta_{1,n-k+1}-2
\dgn{n-k+1}\right)$.
Therefore
\begin{align*}
\dgp{n}{x}&=2\frac{\cnq{n}}{\cnq{n+1}}\sum_{k=0}^n\binomcq{n+1}{k}
\left(\delta_{1,n-k+1}-\dgn{n-k+1}\right)\dbp{k}{x}\nn\\
&=2\dbp{n}{x}-2\frac{\cnq{n}}{\cnq{n+1}}\sum_{k=0}^n\binomcq{n+1}{k}
\dgn{n-k+1}\dbp{k}{x},
\end{align*}
and, by the fact that $\dgn{1}=1$, we obtain the first statement of the 
theorem.
From another side, by Proposition~\eqref{p5Prop::prop4} with $x=1$, we 
obtain
\begin{align*}
b_k&=\frac{\cnq{n}}{\cnq{k}\cnq{n-k+1}}\left(\dgp_{n-k+1}{1}
-\dgn{n-k+1}\right)\\
&=\frac{\cnq{n}}{\cnq{k}\cnq{n-k+1}}\left(\sum_{m=0}^{n-k+1}
\binomcq{n-k+1}{n+1-k-m}\dgn{m}-\dgn{n-k+1}\right)\\
&=\frac{\cnq{n}}{\cnq{k}\cnq{n-k+1}}\sum_{m=0}^{n-k}\frac{\cnq{n-k+1}}
{\cnq{m}\cnq{n-k+1-m}}\dgn{m}\\
&=\frac{\cnq{n}}{\cnq{n+1}}\sum_{m=0}^{n-k}\frac{\cnq{n+1}}{\cnq{k}
\cnq{m}\cnq{n-k+1-m}}\dgn{m}\\
&=\frac{\cnq{n}}{\cnq{n+1}}\sum_{m=0}^{n-k}\binomcq{n+1}{k,m,n+1-k-m}
\dgn{m},
\end{align*}
which completes the proof of the second statement of the theorem.
\end{proof}

\begin{proposition}\label{p5Prop::prop3}
For any polynomial $p(x)\in P$ of degree $n$, there exist constants 
$b_0,b_1,\ldots,b_n$ such that $p(x)=\sum_{k=0}^nb_k\dep{k}{x}$, where 
$b_k=\frac{1}{\cnq{k}}\frac{(\Dcq{x}^kp)(1)-(\Dcq{x}^kp)(0)}{2}$.
\end{proposition}
\begin{proof}
Theorem \eqref{thrm::PolExp}(i) gives 
\begin{align*}
p(x)=\sum_{k=0}^nb_k\dep{k}{x}=\sum_{k=0}^n\frac{\lla\frac{\eqT(t)+1}{2}
t^k\Big|p(x)\rra}{\cnq{k}}\dep{k}{x}.
\end{align*}
Therefore
\begin{align*}
b_k&=\frac{1}{\cnq{k}}\lla\frac{\eqT(t)+1}{2}t^k\Big|p(x)\rra
=\frac{1}{\cnq{k}}\lla\frac{\eqT(t)+1}{2}\Big|t^kp(x)\rra\\
&=\frac{1}{\cnq{k}}\lla\frac{\eqT(t)+1}{2}\Big|\Dcq{x}^kp(x)\rra
=\frac{1}{\cnq{k}}\frac{(\Dcq{x}^kp)(1)+(\Dcq{x}^kp)(0)}{2},
\end{align*}
a required.
\end{proof}

\begin{theorem}
For all integer $n\geq 1$,
\begin{align*}
\dbp{n}{x}=\sum_{k=0}^n\binomcq{n}{k}\dbn{n-k}\dep{k}{x}+\frac{\cnq{n}}
{\cnq{n-1}}\dep{n-1}{x}.
\end{align*}
\end{theorem}
\begin{proof}
Let us assume that $\dbp{n}{x}=\sum_{k=0}^nb_k\dep{k}{x}$. Therefore,
by Proposition~\eqref{p5Prop::prop2}, we have
\begin{align*}
b_k&=\frac{(\Dcq{x}^k\dbp{n}{x})(1)+(\Dcq{x}^k\dbp{n}{x})(0)}{2\cnq{k}}
=\frac{1}{2\cnq{k}}\left(\frac{\cnq{n}}{\cnq{n-k}}\dbp{n-k}{1}+\frac{\cnq{n}}
{\cnq{n-k}}\dbn{n-k}\right).
\end{align*}
So by Proposition \ref{p5prop1}, we obtain
\begin{align*}
b_k&=\frac{1}{2}\binomcq{n}{k}(\dbn{n-k}+\delta_{1,n-k}+\dbn{n-k})
=\frac{1}{2}\binomcq{n}{k}\left(2\dbn{n-k}+\delta_{1,n-k}\right).
\end{align*}
Therefore, we get
\begin{align*}
\dbp{n}{x}&=\sum_{k=0}^n\frac{1}{2}\binomcq{n}{k}\left(2\dbn{n-k}+
\delta_{1,n-k}\right)\dep{k}{x}\\
&=\sum_{k=0}^n\binomcq{n}{k}\dbn{n-k}\dep{k}{x}+\frac{\cnq{n}}
{\cnq{n-1}}\dep{n-1}{x},
\end{align*}
which completes the proof.
\end{proof}

\section{Conclusion}
We defined and studied new analogs of the Bernoulli, Euler, and 
Genocchi polynomials and numbers. Classical identities for them 
including the Euler identity for Bernoulli numbers were extended. 
Moreover, we established connections between these polynomials and 
proved  the formulae which enable to expand other Sheffer-type 
polynomials in terms of degenerate Bernoulli, degenerate Euler, or 
degenerate Genocchi polynomials defined in this work.

\vspace{0.5cm}
\textbf{Acknowledgement}. The research of the first author
was supported  by the Ministry of Science and Technology,
Israel.

\end{document}